\documentclass[12pt, reqno]{amsart}
\usepackage{amsmath}
\usepackage{amssymb}
\usepackage{epsfig}
\usepackage{graphicx}
\usepackage{color}
\definecolor{shadecolor}{gray}{0.875}
\usepackage{amscd}
\usepackage{comment}
\usepackage{enumitem}
\usepackage{mathrsfs}
\usepackage{MnSymbol}

\usepackage{bm}

\usepackage[colorlinks=true]{hyperref}

\numberwithin{equation}{section}

\newcommand{\sho}[1]{{\color{red} \sf $\clubsuit\clubsuit\clubsuit$ Sho: [#1]}}

\input xy
\xyoption{all}

\calclayout
\allowdisplaybreaks[3]

\theoremstyle{plain}
\newtheorem{prop}{Proposition}[section]

\newtheorem{theo}[prop]{Theorem}
\newtheorem{coro}[prop]{Corollary}

\theoremstyle{definition}
\newtheorem{defi}[prop]{Definition}

\newtheorem{conj}[prop]{Conjecture}

\newtheorem{rema}[prop]{Remark}

\def\Mor{\mathrm{Mor}}

\def\Supp{\mathrm{Supp}}

\makeatother
\makeatletter

\author{Brian Lehmann}
\address{Department of Mathematics \\
Boston College  \\
Chestnut Hill, MA \, \, 02467}
\email{lehmannb@bc.edu}

\author{Sho Tanimoto}
\address{Graduate School of Mathematics, Nagoya University, Furocho Chikusa-ku, Nagoya, 464-8602, Japan}
\email{sho.tanimoto@math.nagoya-u.ac.jp}

\title[Free Campana curves]{Fano orbifolds admit free Campana curves}

\begin{document}
\date{\today}

\begin{abstract}
We prove that any Fano orbifold admits a free Campana curve in the sense of Campana. Moreover, assuming that any klt log Fano pair admits a very free rational curve in its smooth locus, we prove that any Fano orbifold is Campana rationally connected.
As an application, we prove the finiteness of the orbifold fundamental groups for Fano orbifolds.
\end{abstract}

\maketitle

\section{Introduction}

We work over an algebraically closed field $\mathbf k$ of characteristic $0$.
Varieties are denoted by $\underline{X}$ and log schemes are denoted by $X$.
Fr\'ed\'eric Campana introduced the study of Campana orbifolds in \cite{Campana04, Campana07}, and this notion plays an important role in arithmetic geometry and birational geometry.
Here is the definition of a Campana orbifold:

\begin{defi}
Let $\underline{X}$ be a normal projective variety defined over $\mathbf k$ and $D = \sum_{i=1}^{n} D_i$ be a reduced divisor on $\underline{X}$.
For each $i$, we pick a positive integer $m_i$ and set
\[
D_\epsilon = \sum_i\left(1 - \frac{1}{m_i}\right)D_i.
\]
When $(\underline{X}, D_\epsilon)$ is klt, we call the pair a {\it klt Campana orbifold}.
\end{defi}

Numerous studies and conjectures in arithmetic and birational geometry naturally fit into this broader context, and there are various recent studies on these, e.g., \cite{BY21, PSTVA, PS24, CLTBT24, CLT26, CLT26toric}.   In the geometric setting, the main object of study is a {\it Campana curve}: a morphism $s : \underline{C} \to \underline{X}$ from a smooth projective curve such that the image is not contained in the support of $D$, is contained in the log smooth locus of $(\underline{X}, D)$, and for every $i$ all intersection points of $\underline{C}$ with $D_i$ have multiplicity $\geq m_i$.

Just as for projective varieties, there is expected to be a close connection between the existence of Campana curves and the positivity of the log anticanonical divisor.  In particular, suppose that the Campana orbifold $(\underline{X}, D_\epsilon)$ is a weak Fano orbifold, i.e.~a klt Campana orbifold such that $-(K_{\underline{X}} + D_\epsilon)$ is big and nef.  The following conjecture is of fundamental importance in this area.

\begin{conj} \label{conj:introconj}
Any weak Fano orbifold is Campana rationally connected.
\end{conj}

A klt Campana orbifold is Campana rationally connected if there is a Campana rational curve $s: \underline{\mathbb{P}^{1}} \to \underline{X}$ passing through two general points on $\underline{X}$.  This is equivalent to asking for the existence of a very free Campana rational curve in the log smooth locus of $(\underline{X}, D)$ as in \cite[Section 4]{CLT26}.  Finding a very free Campana rational curve is a formidable task, as attested by  \cite{Campana04,Campana07, Campana11} and \cite{CLT26, Dand1, Dand2}.

In this paper we prove an analogue of Conjecture \ref{conj:introconj} for curves of higher genus.  The notion of ``relatively $r$-free'' for higher genus Campana curves is analogous to very free for Campana rational curves; see Definition \ref{defi:relativelyrfree}.

\begin{theo}
\label{theo:main}
Let $(\underline{X}, \sum_{i } (1-\frac{1}{m_i})D_i)$ be a weak Fano orbifold. Then there exists a smooth projective curve $\underline{C}$ such that for any $r \geq 0$, there exists a Campana curve $s : \underline{C} \to (\underline{X}, D)^{\mathrm{sm}}$ in the log smooth locus of $(\underline{X}, D)$ such that $s$ is relatively $r$-free.
\end{theo}

One easy way to construct such a curve is the following situation:
suppose that we have a finite cover $f : \underline{Y} \to \underline{X}$ from a smooth projective variety such that for any $i$, $\frac{1}{m_i}f^*D_i$ is an integral reduced divisor and there is no other ramification. Then we have
\[
-K_{\underline{Y}} = -f^*(K_{\underline{X}} + D_\epsilon).
\]
Thus $\underline{Y}$ is a smooth weak Fano variety. In particular, it is rationally connected.
Then since $f^*D_i$ is divisible by $m_i$, any curve on $\underline{Y}$ yields a Campana curve on $\underline{X}$. However, such a finite cover rarely exists in general.

The starting point of this paper is the notion of Seifert bundles which was introduced in \cite{OW75} and recast in \cite{Kollar04, Kollar07} in the context of Sasaki--Einstein metrics.
This construction gives the following fibration: suppose that we have a complex manifold $\underline{X}$ and an irreducible divisor $D \subset \underline{X}$.
Then one can construct a $\mathbb G_m$-torsor $q : \underline{W} \to \underline{X}$ and a degree $m$ sheeted covering $\pi : \underline{Y} \to \underline{W}$ branching along $q^*D$ with multiplicity $m$.  Similar constructions have been introduced and studied under many different names and perspectives: graded rings (\cite{Demazure88}, \cite{Watanabe81}), T-varieties (\cite{AH06}, \cite{AIPSV12}, \cite{LS13}), quotients of Cox rings (\cite{ABHW18}, \cite{BM24}), and orbifold cones (often following Koll\'ar's work, see e.g.~\cite{LL19}). 

Inspired by these constructions, we give a direct and elementary proof of the following theorem:

\begin{theo}[Corollary~\ref{coro:auxilarilylogFano}]
\label{theo:auxilarilylogFanoIntro}
Let $(\underline{X}, \sum_{i=1}^{n}(1-\frac{1}{m_i})D_i)$ be a $\mathbb{Q}$-factorial weak Fano orbifold.
Then there exist a klt log Fano pair $(\underline{Y}, \Delta)$ and a dominant equidimensional projective morphism
\[
f: \underline{Y} \to \underline{X}
\]
such that
\begin{itemize}
\item for $1 \leq i \leq n$, the divisor $\frac{1}{m_i}f^*D_i$ is an integral reduced divisor; and
\item there exists a $\mathbb Q$-Cartier divisor $A$ on $\underline{Y}$ such that $A$ is $f$-relatively ample and we have
\[
-(K_{\underline{Y}} + \Delta) \sim_{\mathbb Q} -f^*\left(K_{\underline{X}} + \sum_i\left(1-\frac{1}{m_i}\right)D_i\right) + A,
\]
which is ample.
\end{itemize}
\end{theo}

Our main theorem can be obtained by combining the previous construction with the following result of \cite{JLR25}.

\begin{theo}[{\cite[Theorem 1.3]{JLR25}}]
Let $(\underline{X}, \Delta)$ be a klt log Fano pair. Then there exists a smooth projective curve $\underline{C}$ such that for any $r \geq 0$, there exists a morphism $s : \underline{C} \to \underline{X}^{\mathrm{sm}}$ to the smooth locus of $\underline{X}$ such that $s$ is relatively $r$-free.
\end{theo}

\begin{rema}
\cite[Theorem 3.6]{JLR25} shows more. It says that any rational ray in the interior of the nef cone is spanned by a relatively very free curve. In particular, for a Fano orbifold $(X, D_\epsilon)$, one can find a family of Campana curves $s : \underline{C} \to \underline{X}$ going through two general points such that $s(\underline{C})$ is strictly nef, i.e., it intersects with any effective divisor positively.
\end{rema}

As an application, we prove the finiteness of orbifold fundamental groups for Fano orbifolds following \cite{JLR25}.  This gives a new proof of a special case of \cite[Theorem 2]{Brauer}.

\begin{theo} \label{intro:orbfundfinite}
Assume that $\mathbf k = \mathbb C$.
Let $(X,D_{\epsilon})$ be a weak Fano Campana orbifold.  Then the orbifold fundamental group $\pi_{1}^{\mathrm{orb}}(X,D_{\epsilon})$ is finite.
\end{theo}

Assuming standard conjectures, we can upgrade Theorem \ref{theo:main} to a statement about rational curves.  Precisely, we need the following well-known conjecture of \cite{KM99}.

\begin{conj}
\label{conj:freerationalcurveonsmoothlocus}
Let $(\underline{X}, \Delta)$ be a klt log Fano pair. Then there exists a smooth projective curve $\underline{C}$ of genus $0$ and a morphism $s : \underline{C} \to \underline{X}^{\mathrm{sm}}$ to the smooth locus of $\underline{X}$ such that $s$ is very free. 
\end{conj}

Assuming this conjecture, we have the following theorem:

\begin{theo}
\label{theo:mainII}
Assume Conjecture~\ref{conj:freerationalcurveonsmoothlocus}. Let $(X, \sum_{i } (1-\frac{1}{m_i})D_i)$ be a weak Fano orbifold. Then $(X, \sum_{i } (1-\frac{1}{m_i})D_i)$ is Campana rationally connected.
\end{theo}

\bigskip

\noindent
{\bf Acknowledgements:}
The authors would like to thank Jihun Park who gave an inspiring talk on Sasaki--Einstein metrics at the conference ``Fano varieties and Related topics'' held at Nagoya University. From his talk, the second author learned about Seifert bundles which led to the construction of auxiliary varieties in this paper. The authors would like to thank Saptarshi Dandapat and Joaqu\'in Moraga for comments on a draft of this paper. In particular, Joaqu\'in provided some helpful references to the literature and pointed out that it should be possible to give a quicker proof of Theorem~\ref{theo:auxilarilylogFanoIntro} using the language of $T$-varieties.

Brian Lehmann was supported in part by Simons Travel Support for Mathematicians, Award Number 851129.
Sho Tanimoto was partially supported by JST FOREST program Grant number JPMJFR212Z and by JSPS KAKENHI Grant-in-Aid (B) 23K25764.

\bigskip

\noindent
{\bf AI disclosure:}
ChatGPT-Astra was used to learn about the construction of Seifert bundles and polish the strategy in this paper. In more details, the development of Theorem~\ref{theo:auxilarilylogFanoIntro} is due to the authors, including the idea to use the construction of Seifert bundles, and Astra was assisting the authors by checking computations. However, the second author erroneously concluded that any Fano orbifold is Campana rationally connected from Theorem~\ref{theo:auxilarilylogFanoIntro}. Then Astra pointed out that singularities of the auxiliary log Fano variety invalidate this claim. AI was unaware of \cite{JLR25}, but of course the authors are aware of this paper. (The first author is an author of \cite{JLR25}.) Our theorems were obtained in this way. Astra also checked the correctness of the paper and pointed out a few minor issues which were fixed by the authors.
Astra was also used to proofread the paper.
The paper itself was written by the authors who are responsible for correctness. 

\section{Background}

We follow the notation of \cite{CLT26}. A variety is an integral separated scheme of finite type over $\mathbf k$ and we denote them by $\underline{X}$. Let $\underline{X}$ be a normal variety over $\mathbf k$ and $D$ be a reduced effective divisor on $\underline{X}$. Then we denote the log scheme associated to $(\underline{X}, D)$ by $(X, D)$ or simply $X$ if there is no confusion.
See \cite{CLT26} and references therein for more details on log schemes.

All projective bundles will denote bundles of quotients.

\subsection{Campana orbifolds and Campana curves}

The following notion was introduced by Campana in his study of special varieties, e.g., \cite{Campana04, Campana07}:

\begin{defi}
Let $\underline{X}$ be a normal projective variety defined over $\mathbf k$ and $D = \sum_{i=1}^{n} D_i$ be a reduced divisor on $\underline{X}$. Let $X$ be the log scheme associated to $(\underline{X}, D)$.
For each $i$, we pick a positive integer $m_i$ and put
\[
D_\epsilon = \sum_i\left(1 - \frac{1}{m_i}\right)D_i.
\]
We assume that $(\underline{X}, D_\epsilon)$ is a klt pair.
We call the pair $(X, D_\epsilon)$ a {\it klt Campana orbifold}.

A klt Campana orbifold $(X, D_\epsilon)$ is called a {\it Fano orbifold} if $-(K_{\underline{X}} + D_\epsilon)$ is ample and {\it weak Fano} if $-(K_{\underline{X}} + D_\epsilon)$ is big and nef.
\end{defi}

Next we introduce the notion of Campana curves:

\begin{defi}
Let $(X, D_\epsilon)$ be a Campana orbifold. Let $s : C \to X$ be a non-degenerate log map in the sense of \cite[Section 2.1]{CLT26} such that the underlying curve $\underline{C}$ is a smooth projective curve and the image is contained in the log smooth locus of $(X, D)$. Let $p_1, \cdots, p_r \in C$ be marked points, and let $\mathbf c_k = (c_{k, i})_i$ be its contact orders as in \cite[Section 2.1]{CLT26}. We assume that there are no non-contact markings.

We say $s$ is a {\it Campana curve} if for every $i, k$, we have either $c_{k, i} \geq m_i$ or $c_{k, i} = 0$.
A Campana curve $s$ is {\it divisible} if additionally for every $i, k$ we have $m_{i} | c_{k,i}$.

We say $s$ is a {\it Campana rational curve} if $C$ has genus $0$.
\end{defi}

The following notion was introduced by Campana:

\begin{defi}
Let $(X, D_\epsilon)$ be a Campana orbifold. We say $(X, D_\epsilon)$ is Campana rationally connected if there is a family of Campana rational curves $\mathcal U \to M$ with the evaluation map $\mathrm{ev} : \mathcal U \to X$ such that
\[
\mathcal U\times_M \mathcal U \to X \times X
\]
is dominant.
\end{defi}

\subsection{Relatively free Campana curves}

The following bundle controls the deformation theory of log maps:

\begin{defi}
Let $(X, D_\epsilon)$ be a Campana orbifold.
Let $s : C \to X$ be a non-degenerate log map such that the underlying curve $\underline{C}$ is a smooth projective curve and the image is contained in the log smooth locus of $(X, D)$. We assume that there are no non-contact markings.
By taking the graph we obtain an induced morphism
\[
\widetilde{s} : C \to X \times \underline{C}.
\]
Let $N_{\widetilde{s}}$ be the normal complex defined in \cite[Section 2.2.2]{CLT26}. This is a locally free sheaf by \cite[Lemma 2.6]{CLT26}.
\end{defi}

We come to the notion of free Campana curves:

\begin{defi} \label{defi:relativelyrfree}
Let $(X, D_\epsilon)$ be a Campana orbifold.
Let $s : C \to X$ be a Campana curve such that the image is contained in the log smooth locus of $(X, D)$. We fix $r \geq 0$.
We say $s$ is {\it relatively $r$-free} if the minimum slope of $N_{\widetilde{s}}$ is greater than or equal to $2g(C) + r$.  We rephrase ``relatively $0$-free'' as ``relatively free'' and ``relatively $1$-free'' as ``relatively very free.''
\end{defi}

\section{Construction of auxiliary log Fano varieties}

Let $\underline{X}$ be a $\mathbb Q$-factorial normal projective variety defined over $\mathbf k$ and $D = \sum_{i = 1}^n D_i$ be a reduced divisor on $\underline{X}$ such that each $D_i$ is a (possibly reducible) reduced divisor on $\underline{X}$. We fix (possibly distinct) multiplicities $m_i \in \mathbb Z_{\geq 1}$
 so that
$(\underline{X}, \sum_i(1-\frac{1}{m_i})D_i)$ is a klt Campana orbifold.  We prove the following theorem:

\begin{theo}
\label{theo:induction}
We fix $1 \leq j \leq n$.
Suppose that we have a klt log Fano pair $(\underline{Y}_j, \Delta)$ and a dominant equidimensional projective morphism
\[
f_j : \underline{Y}_j \to \underline{X}
\]
such that
\begin{itemize}
\item for $1 \leq i < j$, the divisor $\frac{1}{m_i}f_j^*D_i$ is an integral reduced divisor;
\item for $j \leq i \leq n$, $f_j^*D_i$ is an integral reduced divisor;
\item $\Delta \geq \sum_{i = j}^n (1-\frac{1}{m_i})f^*_jD_i$; and
\item there exists a $\mathbb Q$-Cartier divisor $A_j$ on $\underline{Y}_j$ such that $A_j$ is $f_j$-relatively ample and we have
\[
-(K_{\underline{Y}_j} + \Delta) \sim_{\mathbb Q} -f_j^*\left(K_{\underline{X}} + \sum_i\left(1-\frac{1}{m_i}\right)D_i\right) + A_j,
\]
which is ample.
\end{itemize}
Then there exist a klt log Fano pair $(\underline{Y}_{j+1}, \Delta')$ and a dominant equidimensional projective morphism
\[
f_{j+1}: \underline{Y}_{j+1} \to \underline{X}
\]
such that
\begin{itemize}
\item for $1 \leq i < j+1$, the divisor $\frac{1}{m_i}f_{j+1}^*D_i$ is an integral reduced divisor;
\item for $j+1 \leq i \leq n$, $f_{j+1}^*D_i$ is an integral reduced divisor.
\item $\Delta' \geq \sum_{i = j+1}^n (1-\frac{1}{m_i})f^*_{j+1}D_i$; and
\item there exists a $\mathbb Q$-Cartier divisor $A_{j+1}$ on $\underline{Y}_{j+1}$ such that $A_{j+1}$ is $f_{j+1}$-relatively ample and we have
\[
-(K_{\underline{Y}_{j+1}} + \Delta') \sim_{\mathbb Q} -f_{j+1}^*\left(K_{\underline{X}} + \sum_i\left(1-\frac{1}{m_i}\right)D_i\right) + A_{j+1},
\]
which is ample.
\end{itemize}
\end{theo}

\begin{proof}
Let $a$ be a positive integer such that $aD_j$ is Cartier.
Let $B$ be an ample divisor on $\underline{Y}_j$ such that $L = am_jB - af^*_jD_j$ is also ample.
Let $\mathcal L = \mathcal O(L)$. We define
\[
q : \underline{W} = \mathbb P_{\underline{Y}_j}(\mathcal L \oplus \mathcal O) \to \underline{Y}_j.
\]
Let $Y_0$ be the section corresponding to $\mathcal L \oplus \mathcal O \to \mathcal L$ and $Y_\infty$ be the section corresponding to $\mathcal L \oplus \mathcal O \to \mathcal O$.
Let $\xi = \mathcal O(1)$. Then we have
\begin{align*}
&K_{\underline{W}} \sim -2\xi + q^*(K_{\underline{Y}_j} + L) \sim -Y_0 - Y_\infty + q^*K_{\underline{Y}_j}\\
&Y_0 \sim \xi\\
&Y_\infty \sim \xi - q^*L.
\end{align*}
In particular, we have
\[
Y_0 + (am_j-1)Y_\infty + aq^*f^*_jD_j \sim am_jq^*B + am_jY_\infty.
\]
Let $\pi : \underline{Y}_{j+1} \to \underline{W}$ be the normalization of the degree $am_j$ cyclic covering ramified along $Y_0 + (am_j-1)Y_\infty + aq^*f_j^*D_j$.
We denote $\underline{Y}_{j+1} \to \underline{X}$ by $f_{j+1}$.
Choose $\widetilde{Y}_{0},\widetilde{Y}_{\infty},\widetilde{D}_{j}$ so that $\pi^*Y_0 = am_j\widetilde{Y}_0, \pi^*Y_\infty = am_j \widetilde{Y}_\infty, f_{j+1}^*D_j = m_j\widetilde{D}_j$.
Let $\epsilon > 0$ be a sufficiently small rational number.
Let
$$\Delta' = (1-\epsilon)\widetilde{Y}_\infty + \pi^*q^*\Delta - \left(1-\frac{1}{m_j}\right) f^*_{j+1}D_{j}$$
Then we have
\begin{align*}
K_{\underline{Y}_{j+1}} + \Delta' &\sim_{\mathbb Q} \pi^*K_{\underline{W}} +(am_j-1)(\widetilde{Y}_0 + \widetilde{Y}_\infty) +(m_j-1) \widetilde{D}_j + \Delta'\\
&\sim_{\mathbb Q} \pi^*q^*(K_{\underline{Y}_j} + \Delta) -\epsilon \widetilde{Y}_\infty - \frac{1}{am_j}\pi^*\xi\\
& \sim_{\mathbb Q} f_{j+1}^*\left(K_{\underline{X}} + \sum_i\left(1-\frac{1}{m_i}\right)D_i \right) - \epsilon \widetilde{Y}_\infty -\pi^*q^*A_j - \frac{1}{am_j}\pi^*\xi. 
\end{align*}
Thus let $A_{j + 1} = \pi^*q^*A_j + \frac{1}{am_j}\pi^*\xi + \epsilon \widetilde{Y}_\infty$.  By assumption $$-f_{j+1}^{*}\left(K_{\underline{X}} + \sum_i\left(1-\frac{1}{m_i}\right)D_i \right) + \pi^{*}q^{*}A_{j}$$ is the $(q \circ \pi )$-pullback of an ample divisor on $\underline{Y}_{j}$.  Furthermore by construction $\frac{1}{am_{j}}\pi^{*}\xi$ is a big and nef $\mathbb{Q}$-Cartier divisor whose augmented base locus maps finitely to $\underline{Y}_{j}$.  Thus their sum is an ample divisor on $\underline{Y}_{j+1}$.  Then assuming $\epsilon$ is sufficiently small, $-(K_{\underline{Y}_{j+1}} + \Delta' )$ is ample.
Moreover note that $\pi^*q^*A_j + \frac{1}{am_j}\pi^*\xi$ is relatively $f_{j+1}$-ample. Thus for a sufficiently small $\epsilon>0$, $A_{j + 1}$ is relatively $f_{j+1}$-ample.

On the other hand, we have 
\[
K_{\underline{Y}_{j+1}} + \Delta' =\pi^*\left(K_{\underline{W}} +\left(1-\frac{1}{am_j}\right)Y_0 + \left(1-\frac{\epsilon}{am_j}\right)Y_\infty + q^*\Delta\right).
\]
We claim that the pair $$\left(\underline{W},\left(1-\frac{1}{am_j}\right)Y_0 + \left(1-\frac{\epsilon}{am_j}\right)Y_\infty + q^*\Delta\right)$$ is klt.  This is clear away from $Y_{0} \cup Y_{\infty}$ and it also holds true on an open neighborhood of these divisors by inversion of adjunction.  
By \cite[Proposition 5.20]{KM98}, $(\underline{Y}_{j+1}, \Delta')$ is a klt pair.
Thus our assertion follows.
\end{proof}

As a corollary, we have the following statement:  

\begin{coro}
\label{coro:auxilarilylogFano}
Let $(\underline{X}, \sum_{i=1}^{n}(1-\frac{1}{m_i})D_i)$ be a $\mathbb{Q}$-factorial weak Fano orbifold. 
Then there exist a klt log Fano pair $(\underline{Y}, \Delta)$ and a dominant equidimensional projective morphism
\[
f: \underline{Y} \to \underline{X}
\]
such that
\begin{itemize}
\item for $1 \leq i \leq n$, the divisor $\frac{1}{m_i}f^*D_i$ is an integral reduced divisor; and
\item there exists a $\mathbb Q$-Cartier divisor $A$ on $\underline{Y}$ such that $A$ is $f$-relatively ample and we have
\[
-(K_{\underline{Y}} + \Delta) \sim_{\mathbb Q} -f^*\left(K_{\underline{X}} + \sum_i\left(1-\frac{1}{m_i}\right)D_i\right) + A,
\]
which is ample.
\end{itemize}
\end{coro}

\begin{proof} 
Since $(\underline{X}, \sum_i (1-\frac{1}{m_i})D_i)$ is klt weak Fano, there is an effective $\mathbb{Q}$-divisor $E$ such that $(\underline{X}, E + \sum_i (1-\frac{1}{m_i})D_i)$ is klt Fano.  To start, we define $\underline{Y}_1 = \underline{X}$, $f_1$ to be the identity, $\Delta_1 = E + \sum_i(1-\frac{1}{m_i})D_i$, and $A_1 = -E$.   Then a repeated application of Theorem~\ref{theo:induction} yields $f$.
\end{proof}

\section{Existence of free Campana curves}

We prove our main theorems.  The first is a slight strengthening of Theorem \ref{theo:main}.

\begin{theo} \label{theo:divisibleversion}
Let $(\underline{X}, \sum_{i } (1-\frac{1}{m_i})D_i)$ be a weak Fano orbifold. Then there exists a smooth projective curve $\underline{C}$ such that for any $r \geq 0$, there exists a Campana curve $s : \underline{C} \to (\underline{X}, D)^{\mathrm{sm}}$ to the log smooth locus of $(\underline{X}, D)$ such that $s$ is relatively $r$-free and divisible.
\end{theo}

\begin{proof}
Since we can deform relatively $r$-free curves away from codimension $2$ subsets, it suffices to prove the statement after replacing $(\underline{X}, \sum_{i } (1-\frac{1}{m_i})D_i)$ by a $\mathbb{Q}$-factorialization.
Let $(\underline{Y}, \Delta)$ be an auxiliary log Fano pair as in Corollary \ref{coro:auxilarilylogFano} with a dominant morphism $f : \underline{Y} \to \underline{X}$.
By \cite[Theorem 1.3]{JLR25}, we find a relatively $(2g(C) + r)$-free curve $s : \underline{C} \to \underline{Y}^{\mathrm{sm}}$. It follows from \cite[Lemma 3.6]{LRT23} that a general deformation of $\widetilde{s} : \underline{C} \to \underline{Y}^{\mathrm{sm}} \times \underline{C}$ goes through $2g(C) + r + 1$ general points on $\underline{Y}^{\mathrm{sm}} \times \underline{C}$.
Moreover, the relatively free condition implies that for any closed sublocus $\underline{Z} \subset \underline{Y}$ of codimension $\geq 2$ a general deformation of $s$ will be disjoint from $\underline{Z}$.  

Consider the composition $f \circ s : \underline{C} \to \underline{X}$.  Since $f$ is equidimensional, after perhaps replacing $s$ by a general deformation we may assume that $f \circ s(\underline{C})$ lies in the log smooth locus of $(X,D)$ and does not meet any intersection point of two irreducible components of $D$.  Then one can induce the minimal log structure on $\underline{C}$ to obtain a log curve
\[
s' : C \to X.
\]
Since $f^*D_i$ is divisible by $m_i$ and $s(\underline{C}) \subset \underline{Y}^{\mathrm{sm}}$, this composition $s'$ is a divisible Campana curve.  Moreover there is a deformation of $s'$ whose graph goes through any $2g(C) + r + 1$ general points of $X \times \underline{C}$. It follows from the proof of \cite[Lemma 4.7]{CLT26} that a general deformation of $s'$ is relatively $r$-free. Thus our assertion follows.
\end{proof}

\begin{proof}[Proof of Theorem~\ref{theo:main}]
Follows immediately from Theorem \ref{theo:divisibleversion}.
\end{proof}

\begin{proof}[Proof of Theorem~\ref{theo:mainII}]
The proof is similar to the proof of Theorem~\ref{theo:main}.
\end{proof}

\section{Orbifold fundamental groups}

In this section we work over $\mathbb{C}$.  We will study the following notion:

\begin{defi}
Let $(X,D_{\epsilon})$ be a klt Campana orbifold.  The {\it orbifold fundamental group} is defined to be
\begin{equation*}
\pi_{1}^{\mathrm{orb}}(X,D_{\epsilon}) = \pi_{1}(\underline{U}) / \llangle \gamma_{i}^{m_{i}} \rrangle_{i}
\end{equation*}
where $U$ is the complement of $\Supp(D_{\epsilon})$ in the log smooth locus of $(X,D)$, $\gamma_{i}$ is the loop around the irreducible component $D_{i}$, and $\llangle \gamma_{i}^{m_{i}} \rrangle_{i}$ is the normal closure of the subgroup generated by $\gamma_{i}^{m_{i}}$ for all $i$.
\end{defi}

\cite{Brauer} shows that the orbifold fundamental group of a weak log Fano pair is finite.  The argument in this paper gives an alternative proof.  In fact, there are two ways to proceed:

\begin{enumerate}
\item It is well-known that fundamental groups are compatible with taking cones as in local-to-global arguments.  Similarly, the Seifert-type inductive construction in Theorem \ref{theo:induction} is easily seen to satisfy a natural compatibility with orbifold fundamental groups.  In this way one can reduce finiteness to a question concerning the variety $\underline{Y}$.

\item One can directly use the existence of relatively $r$-free families of curves to prove finiteness of the orbifold fundamental group.
\end{enumerate}

As the first approach is well-understood, we will briefly describe the second.

\begin{theo} \label{theo:pi1finite}
Let $(X,D_{\epsilon})$ be a klt Campana orbifold.  Suppose that it admits a relatively $1$-free family of divisible Campana curves.  Then $\pi_{1}^{\mathrm{orb}}(X,D_{\epsilon})$ is finite.
\end{theo}

\begin{proof}
For a non-degenerate log curve $s: C \to X$, we can equip $C$ with an orbifold structure by assigning the multiplicity $1$ to each marked point.  With this convention, the divisibility assumption implies that each morphism $s : (\underline{C}, 0) \to (\underline{X}, D_\epsilon)$ is a divisible orbifold morphism in the sense of \cite[Definition 2.4]{Campana07}.  By the relatively $1$-free assumption, the family of such log curves goes through two general points on $X$.

Forgetting the log structure, we obtain a subvariety $\underline{M} \subset \Mor(\underline{C},\underline{X})$ and a universal family $\mathcal{C} = \underline{M} \times \underline{C}$ with an evaluation map $ev: \underline{\mathcal{C}} \to \underline{X}$ satisfying $\underline{\mathcal{C}} \times_{\underline{M}} \underline{\mathcal{C}} \to \underline{X} \times \underline{X}$ is dominant and contact orders are constant.
In particular, if we fix a general point $x \in \underline{X}$, then we can construct a map $e: \underline{S} \times \underline{C} \to \underline{X}$ satisfying:
\begin{itemize}
\item $\underline{S}$ is smooth irreducible,
\item there is a fiber $S_{1}$ of the projection map $\underline{S} \times \underline{C} \to \underline{C}$ that is contracted to the point $x$ by $e$, and
\item there is a fiber $S_{2}$ of the projection map $\underline{S} \times \underline{C} \to \underline{C}$ such that $e$ maps $S_{2}$ dominantly to $\underline{X}$.
\end{itemize}
After possibly shrinking $\underline{S}$, we may ensure that $\Supp(e^{*}D_{\epsilon})$ consists of disjoint smooth multi sections $T_{1},\ldots,T_{k}$ that meet both $S_{1}$ and $S_{2}$ transversally.  We may also assume that $e(\underline{S}\times \underline{C})$ lies in the locus where $(X,D)$ is log smooth.
Set $\underline{V} = (\underline{S} \times \underline{C}) \backslash (\cup_{j} T_{j})$.  
Then we can identify an exact sequence
\begin{equation*}
1 \to \llangle \sigma_{j} \rrangle \to \pi_{1}(\underline{V})  \to \pi_{1}(\underline{S}) \times \pi_{1}(\underline{C}) \to 1
\end{equation*}
where the $\sigma_{j}$ are loops around the $T_{j}$ and $\llangle - \rrangle$ denotes the normal closure.  
 We can give $(\underline{S} \times \underline{C}, \sum T_{j})$ the structure of a klt Campana orbifold by assigning the multiplicity $1$ to each irreducible component.  Thus we have an identification $\pi_{1}^{\mathrm{orb}}(S\times C,0) \cong \pi_{1}(\underline{S}) \times \pi_{1}(\underline{C})$.

Note that $e: (S\times C,0) \to (X,D_{\epsilon})$ is a divisible orbifold morphism.  Since the domain is smooth and maps into the SNC locus of $(X,D_{\epsilon})$, we may apply \cite[Proposition 11.5]{Campana07} to obtain a natural homomorphism
\[
\pi_{1}(\underline{S}) \times \pi_{1}(\underline{C}) \cong \pi_{1}^{\mathrm{orb}}(S\times C, 0) \to \pi_{1}^{\mathrm{orb}}(X,D_{\epsilon}).
\]
The images of $\pi_{1}(\underline{S}_{1})$ and $\pi_{1}(\underline{S}_{2} )$ are conjugate to each other in $\pi_{1}(\underline{S}) \times \pi_{1}(\underline{C})$.
 Since $e : \underline{S}_2 \to \underline{X}$ is dominant, we have that $e_{*}\pi_{1}(V \cap \underline{S}_{2})$ is a finite index subgroup of $\pi_{1}(\underline{X}^{sm} \backslash \mathrm{Supp}(D_\epsilon))$.  Passing to orbifold fundamental groups by taking quotients, we see that $e_{*}\pi_{1}(\underline{S}_{2})$ is a finite index subgroup of $\pi_{1}^{\mathrm{orb}}(X,D_{\epsilon})$.  However, since $e$ contracts the section $\underline{S}_{1}$ to a point, by the conjugacy relation proved earlier $e_{*}\pi_{1}(\underline{S}_{2})$ is trivial. Thus $\pi_{1}^{\mathrm{orb}}(X,D_{\epsilon})$ is finite.
\end{proof}

\begin{proof}[Proof of Theorem \ref{intro:orbfundfinite}:]
Combine Theorem \ref{theo:divisibleversion} and Theorem \ref{theo:pi1finite}.
\end{proof}

\bibliographystyle{alpha}
\bibliography{CampanaCurveIII}

\end{document}